\documentclass{amsart}

\usepackage{amssymb}
\usepackage{amsmath,amsthm}
\usepackage[dvips]{graphicx}
\usepackage{wrapfig}
\usepackage{bm}

\newtheorem{theorem}{Theorem}[section]
\newtheorem{corollary}[theorem]{Corollary}
\newtheorem{lemma}[theorem]{Lemma}

\newtheorem{question}[theorem]{Question}
\theoremstyle{definition}
\newtheorem{definition}[theorem]{Definition}
\newtheorem{example}[theorem]{Example}
\theoremstyle{remark}
\newtheorem{remark}[theorem]{Remark.}
\theoremstyle{conjecture}
\newtheorem{conjecture}[theorem]{Conjecture.}

\title{Up-down colorings of virtual-link diagrams and the necessity of Reidemeister moves of type I\hspace{-1.5pt}I}
\author{Kanako Oshiro}
\address{Deparment of Information and Communication Sciences, Sophia University, 7-1 Kioicho, Chiyoda-ku, Tokyo 102-8554, Japan}
\email{oshirok@sophia.ac.jp}
\thanks{The first author was supported by JSPS KAKENHI Grant Number 16K17600.}
\author{Ayaka Shimizu}
\address{Department of Mathematics, National Institute of Technology, Gunma College, 580 Toriba-cho, Maebashi-shi, Gunma 371-8530, Japan}
\email{shimizu@nat.gunma-ct.ac.jp}
\thanks{The second author was supported by Grant for Basic Science Research Projects from The Sumitomo Foundation (160154).}
\author{Yoshiro Yaguchi}
\address{Department of Mathematics, National Institute of Technology, Gunma College, 580 Toriba-cho, Maebashi-shi, Gunma 371-8530, Japan}
\email{yaguchi-y@nat.gunma-ct.ac.jp}
\thanks{}

\date{}

\begin{document}

\maketitle

\begin{abstract}
We introduce an up-down coloring of a virtual-link diagram.
The colorabilities give a lower bound of the minimum number of Reidemeister moves of type I\hspace{-1.5pt}I which are needed between two $2$-component virtual-link diagrams.
By using the notion of a quandle cocycle invariant, we determine the necessity of Reidemeister moves of type I\hspace{-1.5pt}I for a pair of diagrams of the trivial virtual-knot. This implies that for any virtual-knot diagram $D$, there exists a diagram $D'$ representing the same virtual-knot such that any sequence of generalized Reidemeister moves between them includes at least one Reidemeister move of type  I\hspace{-1.5pt}I.
\end{abstract}

\section{Introduction}

Necessity of Reidemeister moves, or estimations of the minimum number of those moves, between two diagrams of the same classical-knot or classical-link have been studied by several ways; for example, for studies using the geometrical feature of diagrams or their Gauss diagram, see \cite{hagge, hass, hayashi1, hayashi2, manturov, ostlund, trace}, for studies using the concepts of quandle colorings and quandle cocycle invariants, see \cite{CESS, CG}. 

In terms of Reidemeister moves of type I\hspace{-1pt}I, by using geometric properties, T.~J.~Hagge  \cite{hagge} and V.~O.~Manturov  \cite{manturov} costructed pairs of diagrams of the same classical-knot such that at least one Reidemeister moves of type I\hspace{-1pt}I is needed between them. 
Especially, Hagge proved that every classical-knot admits a pair $(D, D')$ of diagrams such that any sequence of Reidemeister moves between $D$ and $D'$ includes a Reidemeister move of type I\hspace{-1pt}I. 
On the other hand, as far as we know, there is no method using the concept of quandle colorings or quandle cocycle invariants to determine the necessity of Reidemeister moves of type I\hspace{-1pt}I.
\footnote{Z. Cheng and H. Gao \cite{CG} also studied about  the necessity of Reidemeister moves of type I\hspace{-1pt}I by using the concept of quandle colorings. However, their method contains a key error in the conditions of an algebraic structure.}

In this paper, we study about the necessity (or an estimation of the minimum number) of Reidemeister moves of type I\hspace{-1pt}I for virtual-link diagrams by using up-down colorings and the notion of quandle cocycle invariants.
Note that our method is also useful for classical-link diagrams, that is, the readers may read all the part  except for Theorem~\ref{thm:4} by replacing  ``virtual" with ``classical" and ``generalized Reidemeister moves" with ``Reidemeister moves". 
Besides, in this paper, a virtual-link (or virtual-knot) diagram means an oriented virtual-link (or virtual-knot) diagram.

For  a virtual-link (or virtual-knot) diagram $D$, we introduce the {\it number of $n$-up-down colorings} $\# {\rm Col}_n(D)$ which is analogous to quandle coloring numbers, the {\it maximum order of up-down colorings} maxord$(D)$, and  a multi-set $\Phi_f(D)$ which is analogous to quandle cocycle invariants. We show the following theorems:
\begin{theorem}\label{thm:1}
Let $D$ and $D'$ be diagrams which represent the same virtual-link. 
If $\# {\rm Col}_{n}(D)\not = \# {\rm Col}_{n}(D')$, then any sequence of generalized Reidemeister moves between $D$ and $D'$ includes at least one Reidemeister move of type I\hspace{-1pt}I.
\end{theorem}
\begin{theorem}\label{thm:2}
Let $D$ and $D'$ be $2$-component virtual-link diagrams which represent the same virtual-link.
Any finite sequence of generalized Reidemeister moves between $D$ and $D'$ includes at least $|{\rm maxord}(D)-{\rm maxord}(D')|/2$ Reidemeister moves of type I\hspace{-1pt}I.
\end{theorem}
\begin{theorem}\label{thm:3}
Let $D$ and $D'$ be diagrams which represent the same virtual-knot. 
If $\Phi_f(D)\not=\Phi_f(D')$ as multi-sets, then any sequence of generalized Reidemeister moves between $D$ and $D'$ includes at least one Reidemeister move of type I\hspace{-1pt}I.
\end{theorem}
\noindent As an application of Theorem~\ref{thm:3}, we have the following result.
\begin{theorem}\label{thm:4}
For any virtual-knot diagram $D$, there exists a diagram $D'$ representing the same virtual-knot such that any sequence of generalized Reidemeister moves between $D$ and $D'$ includes at least one Reidemeister move of type I\hspace{-1pt}I. 
\end{theorem}

The rest of the paper is organized as follows: In Section 2, we review the definitions of virtual-links and the generalized Reidemeister moves. In Section 3, we introduce up-down colorings for virtual-link diagrams and prove Theorems~\ref{thm:1} and \ref{thm:2}. 
In Section 4, we define a multi-set $\Phi_f(D)$ which is analogous to quandle cocycle invariants  and prove Theorems~\ref{thm:3} and \ref{thm:4}. 

\section{Virtual-links}

A {\it classical-link}  with  $r$-components  is $r$ circles embedded  in $\mathbb{R}^3$ ($r=1,2, \dots $).  
We call  a classical-link with $r=1$ a {\it classical-knot}. 
A {\it diagram} of a classical-link $L$ is the image $p(L)$ of $L$ by a regular projection $p: \mathbb R^3 \to \mathbb R^2$ with over/under information at each crossing.
We call such a crossing a {\it real-crossing}, see the left picture of Figure \ref{crossing}.
\begin{figure}[h]
\begin{center}
\includegraphics[width=40mm]{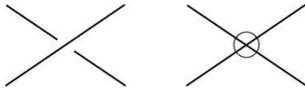}
\caption{A real-crossing and a virtual-crossing.}
\label{crossing}
\end{center}
\end{figure}
It is known that two classical-link diagrams represent the same classical-link if and only if there exists a finite sequence of the Reidemeister moves of type I, type I\hspace{-0.5pt}I or type I\hspace{-0.5pt}I\hspace{-0.5pt}I between them, where the Reidemeister moves are the local transformations (two RIs, an RI\hspace{-0.5pt}I and  two RI\hspace{-0.5pt}I\hspace{-0.5pt}Is) on classical-link diagrams depicted in Figure \ref{g-reidemeister}, see \cite{reidemeister}. 
\begin{figure}[h]
\begin{center}
\includegraphics[width=120mm]{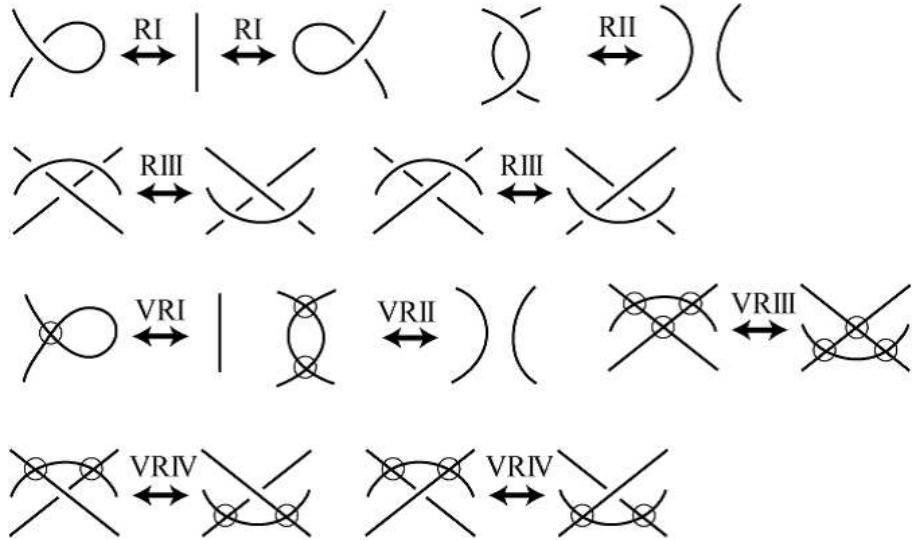}
\caption{The Generalized Reidemeister moves.}
\label{g-reidemeister}
\end{center}
\end{figure}

A {\it virtual-link diagram} is generic closed curves in $\mathbb R^2$
each of whose double points  is a real-crossing  or a virtual-crossing depicted in Figure \ref{crossing}. 
We call each closed curve a {\it component} of the virtual-link diagram, and a virtual-link diagram with $r$-components an {\it $r$-component virtual-link diagram}.
A $1$-component virtual-link diagram is also called   a {\it virtual-knot diagram}.
It is said that two virtual-link diagrams {\it represent the same virtual-link} if there exists a finite sequence of the generalized Reidemeister moves which are depicted in Figure \ref{g-reidemeister} between them. 
In this way, classical-links are expanded to virtual-links. 
For more details, see \cite{kauffman}. 

In this paper, we assume that virtual-link diagrams are oriented, and then,  generalized Reidemeister moves mean oriented generalized Reidemeister moves.
We call a  Reidemeister move of type I, of type I\hspace{-0.5pt}I and  of  type I\hspace{-0.5pt}I\hspace{-0.5pt}I   an {\it RI-move}, an {\it RI\hspace{-0.5pt}I-move} and an {\it RI\hspace{-0.5pt}I\hspace{-0.5pt}I-move}, respectively.
Similarly, for the other generalized Reidemeister moves, we call them a {\it VRI}-move, a {\it VRI\hspace{-0.5pt}I}-move, a {\it VRI\hspace{-0.5pt}I\hspace{-0.5pt}I}-move and a {\it VRI\hspace{-0.5pt}V}-move, respectively.  

\section{Up-down colorings and the necessity of Reidemeister moves of type I\hspace{-1pt}I}
 We note again that we may read this section by replacing  ``virtual" with ``classical" and ``generalized Reidemeister moves" with ``Reidemeister moves".

In this section, we define an up-down coloring for a virtual-link diagram and investigate its properties. 

Let $n$ be a positive integer and $\mathbb Z_n$ the cyclic group $\mathbb Z/ n \mathbb Z$. 

Let $D$ be a virtual-link diagram. 
The diagram $D$ is separated into the small edges such that the end points of each edge are real-crossings and there is no real-crossing in the interior of each edge.
We call such an edge a {\it semi-arc} of $D$.   
Let $\mathcal{SA}(D)$ denote the set of semi-arcs of $D$.

\begin{definition}
An \textit{$n$-up-down coloring} of $D$ is a map $C: \mathcal{SA}(D) \to \mathbb Z_n$ which  satisfies the following condition:
For a real-crossing $c$ of $D$, let $u_1$, $u_2$ (resp. $o_1$, $o_2$) be the under-semi-arcs (resp. over-semi-arcs) around $c$  such that the orientation of $D$ points from $u_1$ to $u_2$  (resp. from $o_1$ to $o_2$). Then it holds that  
\begin{align}\label{condi1}
C(u_2)=C(u_1)-1 \mbox{\  and \ } C(o_2)=C(o_1)+1.
\end{align}
See Figure~\ref{up-down}. 
\begin{figure}[h]
\begin{center}
\includegraphics[width=60mm]{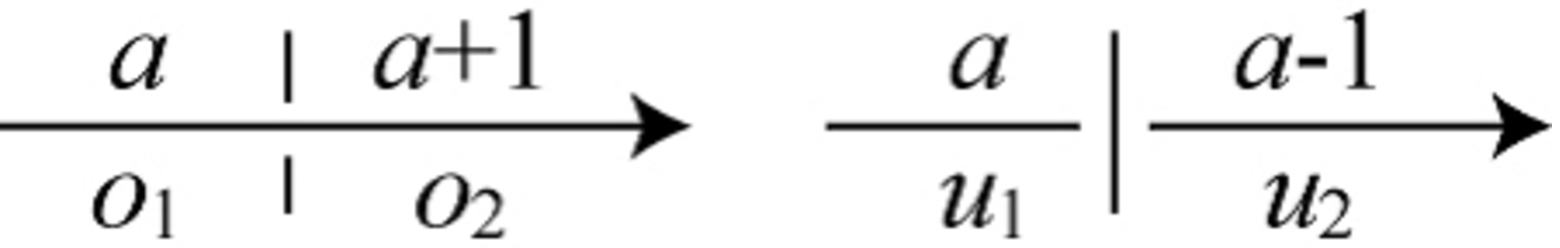}
\caption{The conditions of an up-down coloring.}
\label{up-down}
\end{center}
\end{figure}
When we do not specify $n$, we also call it an {\it up-down coloring} of $D$ for simplicity.
For each semi-arc $e$ of $D$, we call $C(e)$ the {\it color} of $e$.
\end{definition}

\begin{remark}
Up-down colorings are regarded as a generalization of the warping degree labellings defined in \cite{shimizu}, and see also \cite{lowrance}. 
\end{remark}

We denote by ${\rm Col}_{n}(D)$ the set of $n$-up-down colorings of $D$. 
Since $\mathcal{SA}(D)$ and $\mathbb Z_n$ are finite sets, so is ${\rm Col}_{n}(D)$.
Therefore, we call the cardinality of ${\rm Col}_{n}(D)$ the \textit{number of $n$-up-down colorings} of $D$, and denote it by $\#{\rm Col}_{n}(D)$.

Now, we consider about  the numbers of $n$-up-down colorings in the cases of virtual-``knots''.
\begin{lemma}\label{coloringnumber}
For any virtual-knot diagram $D$, 
$
\# {\rm Col}_{n}(D) =n.
$
\end{lemma}
\begin{proof}
Choose a semi-arc $e$ of $D$ and fix it. 
We pass through the over-crossings  as many as the under-crossings while we travel along the diagram $D$ from $e$ to $e$ according to the orientation of $D$.
This implies that for any $a \in \mathbb{Z}_n$, 
we have a unique up-down coloring of $D$ such that the color of $e$ is $a$. 
Therefore the number of $n$-up-down colorings of $D$ coincides with that of choices of elements of $\mathbb Z_n$.
\end{proof}
\noindent Moreover, we have the following lemma: 
\begin{lemma}
Let $D$ be a virtual-knot (or virtual-link) diagram and $C$ an $n$-up-down coloring of $D$.
Set  a map $C': \mathcal{SA}(D) \to \mathbb Z_n$  by  $C'(e)=C(e)+1$ for $e \in  \mathcal{SA}(D)$. 
Then the  map $C'$ is also an $n$-up-down coloring of $D$. 
\label{plus-one}
\end{lemma}
\begin{proof}
By using the equation (\ref{condi1}) of the condition of an up-down coloring, we have 
$$C'(u_2)= C(u_2)+1 = (C(u_1)-1)+1= (C(u_1)+1) -1= C'(u_1)-1$$
and
$$C'(o_2)= C(o_2)+1 = (C(o_1)+1)+1= C'(o_1)+1.$$
Therefore  $C'$ also satisfies the condition (\ref{condi1}) of an up-down coloring.
\end{proof}
\noindent By Lemmas~\ref{coloringnumber} and \ref{plus-one}, next property holds:
\begin{corollary}\label{cor:shift}
For a virtual-knot diagram $D$ and an $n$-up-down coloring $C$ of $D$,
\[
{\rm Col}_{n}(D)= \{ C+i ~|~ i \in \mathbb Z_n \},
\] 
where $C+i$ is the map from $\mathcal{SA}(D)$to $ \mathbb Z_n$ which maps a semi-arc $e$ to $ C(e)+i$.
\end{corollary}

Next, let us consider about the numbers of $n$-up-down colorings   in the cases of virtual-``links''.
In the cases of virtual-links with at least $2$-components, the number of $n$-up-down colorings depends on the choice of a diagram $D$. 
For example, $T(1)$ does not have a $4$-up-down coloring, while $T(2)$ does,  where $T(1)$ and $T(2)$ are the virtual-link diagrams, representing the same virtual-link, depicted in Figure~\ref{ti}.
However, Theorem~\ref{thm:1} shows that the numbers of $n$-up-down colorings are useful to detect the necessity of RI\hspace{-1.5pt}I-moves.

\phantom{x}
\noindent {\it Proof of Theorem~\ref{thm:1}.}
It suffices to show that the number of $n$-up-down colorings of  a virtual-link diagram is unchanged  under the generalized Reidemeister moves except for the RI\hspace{-1.5pt}I-moves.

Let $D$ and $D'$ be virtual-link diagrams such that $D'$ is obtained from $D$ by a single generalized Reidemeister move other than the RI\hspace{-1pt}I-moves. Let $E$ be a $2$-disk in $\mathbb R^2$ in which the move  is applied. Let $C$ be an $n$-up-down coloring of $D$. 
We define an  $n$-up-down coloring $C'$ of $D'$, corresponding to $C$, by $C' (e) = C(e)$ for a semi-arc $e$  appearing in $\mathbb R^2- E$.
Then the colors of the semi-arcs  appearing  in $E$, by $C'$, are uniquely determined, see Figure \ref{r1r3} for the RI-moves, Figure \ref{r3} for the RI\hspace{-1pt}I\hspace{-1pt}I-moves and Figure \ref{vr4} for the VRI\hspace{-1pt}V-moves. 
Thus we have a bijection
\[
{\rm Col}_{n}(D) \to {\rm Col}_{n}(D); C \mapsto C'.
\]
\hfill \qed

\phantom{x}

A virtual-link diagram $D$ is {\it $n$-up-down colorable} if there exists an $n$-up-down coloring  of $D$. 
The {\it maximum order of up-down colorings}  of a virtual-link digram $D$ is the maximum number of  positive integers $n$ such that $D$ is $n$-up-down colorable if it is finite, and $0$ otherwise.
We denote it by  ${\rm maxord}(D)$. 
By Theorem~\ref{thm:2}, we give an estimation of the minimum number of required RI\hspace{-1pt}I-moves between given two $2$-component virtual-link diagrams. 

\phantom{x}
\noindent {\it Proof of Theorem~\ref{thm:2}.}
By the proof of Theorem~\ref{thm:1}, it is easy to show that  
the generalized Reidemeister moves other than the RI\hspace{-1pt}I-moves do not change the maximum order of up-down colorings.
Hence, it is sufficient to show that for virtual-link diagrams $D$ and $D'$ which are related by a single RI\hspace{-1pt}I-move,  the value $|{\rm maxord}(D)-{\rm maxord}(D')|/2$ is at most one.

Let $D$ and $D'$ be virtual-link diagrams such that $D'$ is obtained from $D$ by the  RI\hspace{-1pt}I-move shown in Figure~\ref{e-disk}. 
Let $E$ be a $2$-disk in $\mathbb R^2$ in which the move  is applied. 
As in Figure~\ref{e-disk}, we denote by $e_1$ and $e_2$ (or $e_1'$ and $e_2'$) two semi-arcs of $D$ (or of $D'$) appeared in $E$.
Here, we assume that $e_1$ and $e_2$ (or $e_1'$ and $e_2'$) are in the distinct components.
By traveling along one component of the diagram $D$ from $e_1$ to $e_1$ according to the orientation of $D$, we read the colors of the semi-arcs which are passed through.
Thus we can see that $D$ is $n$-up-down colorable if and only if $C(e_1)= C(e_1)+o-u$ in $\mathbb Z_n$ (and $C(e_2)= C(e_2)-o+u$ in $\mathbb Z_{n}$) for some $n$-up-down coloring $C$, see Figure~\ref{e-disk}, where $o$ and $u$ is the numbers of the non-self  over-crossings and the non-self  under-crossings, respectively, which we passed through.
This implies that ${\rm maxord}(D)=|o-u|$. 
On the other hand, the numbers of the non-self  over-crossings and the non-self  under-crossings which are passed through while we travel along one component of the diagram $D'$ from $e_1'$ to $e_1'$ according to the orientation of $D'$ are $o$ and $u+2$, respectively.
Thus ${\rm maxord}(D')=|o-(u+2)|=|o-u-2|$ holds. 
Hence we have 
\[
 {\rm maxord}(D)-2 =|o-u |-2\leq |o-u-2|= {\rm maxord}(D'),
\]
and 
\[
{\rm maxord}(D')=|o-u-2|\leq |o-u| +2 = {\rm maxord}(D)+2.
\]
Therefore it holds that 
\[
| {\rm maxord}(D)- {\rm maxord}(D')|/2\leq 1.
\]
Similarly, we can see that in the case that $e_1$ and $e_2$ (or $e_1'$ and $e_2'$) are in the same component, we have 
\[
| {\rm maxord}(D)- {\rm maxord}(D')|/2=0\leq 1.
\]

For the other RI\hspace{-1pt}I-moves, we also have the same inequality.
\begin{figure}[h]
\begin{center}
\includegraphics[width=90mm]{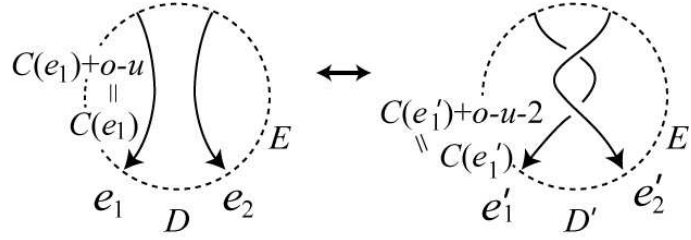}
\caption{An RI\hspace{-1pt}I-move and colorability by up-down colorings.}
\label{e-disk}
\end{center}
\end{figure}
\hfill \qed

\phantom{x}

\begin{example}
Let $\{T(i)\}_{i \in \{0, 1,2, \ldots \}}$ be  the family of the virtual-link diagrams depicted in Figure~\ref{ti}.
When  $i\in \{0, 1,2, \ldots  \}$, $T(i)$ is $n$-up-down colorable if and only if $2i \equiv 0 \pmod{n}$, see Figure~\ref{ti-label}.
Hence, the maximum order of up-down colorings of $T(i)$ is 
\[
{\rm maxord}(T(i))=2i.
\]
Hence, by Theorem~\ref{thm:2}, we can see that 
at least $|i-j|$ RI\hspace{-1.5pt}I-moves are needed to transform $T(i)$ to $T(j)$ by using generalized Reidemeister moves.  
Indeed, $|i-j|$ coincides with the minimum number of RI\hspace{-1pt}I-moves needed for transformations between $T(i)$ and $T(j)$.
\begin{figure}[h]
\begin{center}
\includegraphics[width=110mm]{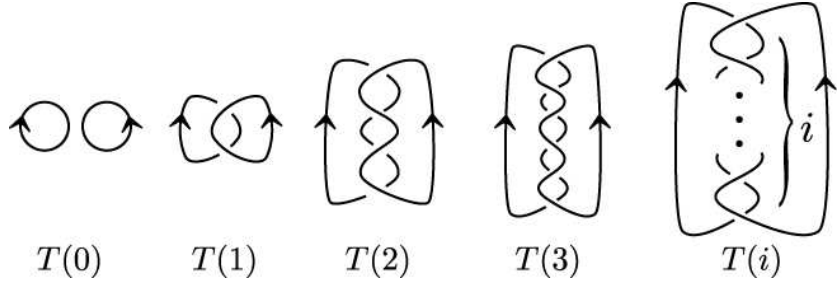}
\caption{The $2$-component virtual-link diagrams $T(i)$.}
\label{ti}
\end{center}
\end{figure}
\begin{figure}[h]
\begin{center}
\includegraphics[width=60mm]{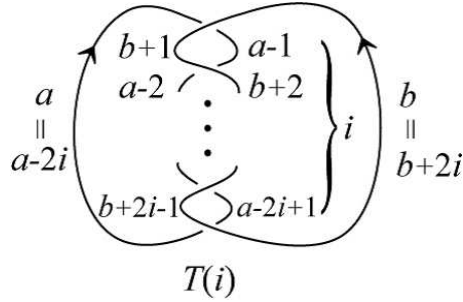}
\caption{Colorability of $T(i)$ by up-down colorings..}
\label{ti-label}
\end{center}
\end{figure}
\end{example}

\begin{example}
For the virtual-link diagrams $T(5)$ and $T'(5)$ in Figure \ref{t5}, the maximum orders of up-down colorings  are 
\[
{\rm maxord}(T(5))=10   \mbox{ and } {\rm maxord}(T'(5))=6.
\]
Therefore at least $2(=|10-6|/2)$ RI\hspace{-1pt}I-moves are needed to transform $T(5)$ to $T'(5)$ by using generalized Reidemeister moves.  
Indeed, $2$ coincides with the minimum number of RI\hspace{-1pt}I-moves needed for transformations between $T(5)$ and $T'(5)$.

Similarly, we can see that the minimum number of RI\hspace{-1pt}I-moves needed for transformations between $T(5)$ and $T''(5)$ (resp. $T'(5)$ and $T''(5)$) is $4$ (resp. $2$).
\begin{figure}[h]
\begin{center}
\includegraphics[width=90mm]{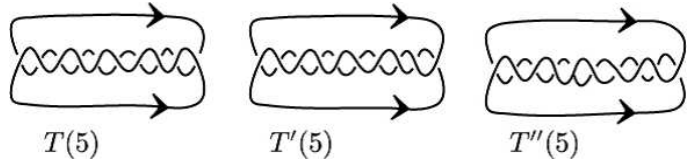}
\caption{The $2$-component virtual-link diagrams $T(5)$, $T'(5)$ and $T''(5)$.}
\label{t5}
\end{center}
\end{figure}
\end{example}

\begin{remark}
The minimum number of required RI\hspace{-1.5pt}I-moves between given two  $2$-component virtual-link diagrams can be also estimated by using the number of non-self real-crossings  since RI\hspace{-1.5pt}I-moves only change the number of  such real-crossings. 
More precisely, the half of the difference of the numbers of non-self real-crossings gives a lower bound of the number of required RI\hspace{-1.5pt}I-moves.
Hence, it is easily seen that 
two different diagrams $T(i)$ and $T(j)$ in the family $\{T(i)\}_{i\in \{0, 1,2, \ldots \}}$ of Example~\ref{ti} need at least $|i-j|$ RI\hspace{-1.5pt}I-moves to be transformed each other.  
However the minimum number  of RI\hspace{-1.5pt}I-moves needed for transformations between $T(5)$ and $T'(5)$ (or $T'(5)$ and $T''(5)$, $T(5)$ and $T''(5)$) in Figure~\ref{t5}  can not be detected by using the number of non-self real-crossings, but by using our method.  
%
\end{remark}

\section{Cocycle invariants using up-down colorings and the necessity of Reidemeister moves of type I\hspace{-1pt}I}\label{section:cocycleinvariants}
We note again that we may read this section except for the proof of Theorem~\ref{thm:4} by replacing  ``virtual" with ``classical" and ``generalized Reidemeister moves" with ``Reidemeister moves".

In this section, we define  a multi-set $\Phi_f(D)$ which is analogous to quandle cocycle invariants.

Let $n$ be a positive integer. 
Let $A$ be an abelian group.

\begin{definition}
An {\it $n$-up-down cocycle} is a map 
$f : \mathbb{Z}_n \times \mathbb{Z}_n \times \{ +, - \} \to A$ satisfying the following conditions: 
\begin{itemize}
\item For any $a \in \mathbb{Z}_n$ and $\varepsilon \in \{ +, - \}$, 
\begin{itemize}
\item[(0)] $f(a, a, \varepsilon )=0$.
\end{itemize} 
\item For any $a, b, c \in \mathbb{Z}_n$, 
\begin{itemize}
\item[(1)] $f(a-1, b, -)+f(b+1, c+1, +)+f(a-1, c+2, +)=f(a-2, b-1, -)+f(b, c+2, +)+f(a, c+1, +)$,
\item[(2)] $f(a-1, b, -)+f(b, c+1, -)+f(a-2, c, -)=f(a-2, b-1, -)+f(b-1, c, -)+f(a-1, c+1, -)$,
\item[(3)] $f(a-1, b+1, +)+f(b+1, c+1, +)+f(a-1, c+1, -)=f(a, b, +)+f(b, c+2, +)+f(a-2, c, -)$,
\item[(4)] $f(a-1, b+1, +)+f(b, c+1, -)+f(a, c+1, +)=f(a, b, +)+f(b-1, c, -)+f(a-1, c+2, +)$,
\item[(5)] $f(a, b+1, +)+f(b+1, c+2, +)+f(a-1, c+1, +)=f(a-1, b, +)+f(b, c+1, +)+f(a, c+2, +)$,
\item[(6)] $f(a, b+1, +)+f(b, c, -)+f(a-2, c+1, -)=f(a-1, b, +)+f(b-1, c+1, -)+f(a-1, c, -)$,
\item[(7)] $f(a-2, b, -)+f(b+1, c+2, +)+f(a-1, c, -)=f(a-1, b-1, -)+f(b, c+1, +)+f(a-2, c+1, -)$ and 
\item[(8)] $f(a-2, b, -)+f(b, c, -)+f(a, c+2, +)=f(a-1, b-1, -)+f(b-1, c+1, -)+f(a-1, c+1, +)$.
\end{itemize}
\end{itemize}
When we do not specify $n$, we call such a map an {\it up-down cocycle} for simplicity.
\end{definition}
\noindent 
We will later show that the above conditions are related to the RI-moves and the RI\hspace{-1.5pt}I\hspace{-1.5pt}I-moves, see the proof of Theorem~\ref{thm:3}. 
On the other hand, quandle cocycle conditions are also related to the RI-moves and  an RI\hspace{-1.5pt}I\hspace{-1.5pt}I-move, see \cite{CJKLS}. 
In that sense, we call a map satisfying the above conditions an up-down cocycle.
However we do not know if there exists a (co)homology theory which are related to up-down cocycles.

\begin{definition}
An $n$-up-down cocycle $f$ is {\it shiftable} if 
$f(a+1, b+1, \varepsilon )=f(a, b, \varepsilon )$ for any $a, b \in \mathbb{Z}_n$ and $\varepsilon \in \{ +, - \}$. 
\end{definition}

\begin{remark}
A map $f : \mathbb{Z}_n \times \mathbb{Z}_n \times \{ +, - \} \to A$ is a shiftable $n$-up-down cocycle if and only if it satisfies the following conditions (A)-(C):
\begin{itemize}
\item[(A)] For any $a \in \mathbb Z_n$ and $\varepsilon \in \{+, -\}$, $f(a,a,\varepsilon)=0$.
\item[(B)] For any $a,b \in \mathbb Z_n$ and $\varepsilon \in \{ +, -\}$, $f(a+1, b+1, \varepsilon)= f(a, b, \varepsilon)$.
\item[(C)] For any $a,b,c \in \mathbb Z_n$ and $\varepsilon \in \{ +, -\}$, the following equations (i)-(iii) hold:
\begin{itemize}
\item[(i)] $f(b+1, c+1, \varepsilon)+f(a-1, c+2, \varepsilon)=f(b, c+2, \varepsilon)+f(a, c+1, \varepsilon)$,
\item[(ii)] $f(a-1, b+1, \varepsilon)+f(b+1, c+1, \varepsilon)=f(a, b, \varepsilon)+f(b, c+2, \varepsilon)$ and
\item[(iii)] $f(a-1, b+1, \varepsilon)+f(a, c+1, \varepsilon)=f(a, b, \varepsilon)+f(a-1, c+2, \varepsilon)$.
\end{itemize}
\end{itemize}
\end{remark}
\begin{example}\label{exam:cocycle}
Define $f: \mathbb Z_4 \times \mathbb Z_4 \times \{+,-\} \to \mathbb Z_4$ by 
\phantom{x}
\begin{eqnarray*}
f(a,b,+)=\left\{ \begin{array}{ll}
0 & (a=b), \\
2 & (a=b+1),\\
2 & (a=b+2),\\
0 & (a=b+3)\mbox{ and}\\
\end{array} \right.
\end{eqnarray*}
\begin{eqnarray*}
f(a,b,-)=\left\{ \begin{array}{ll}
0 & (a=b), \\
1 & (a=b+1),\\
2 & (a=b+2),\\
3 & (a=b+3).\\
\end{array} \right.
\end{eqnarray*}
\noindent Then $f$ is a shiftable $4$-up-down cocycle. 
\end{example}

Let $D$ be a diagram of a virtual-knot and $C$ an $n$-up-down coloring of $D$.
Let $f: \mathbb Z_n \times \mathbb Z_n \times \{+,-\} \to A$ be an $n$-up-down cocycle.
For each real-crossing $c$ of $D$, we define the \textit{weight} $w_c$ of $c$ as follows:
Let $u_1$, $u_2$ (resp. $o_1$, $o_2$) be the under-semi-arcs (resp. over-semi-arcs) around $c$  such that the orientation of $D$ points from $u_1$ to $u_2$  (resp. from $o_1$ to $o_2$). 
\begin{itemize}
\item When $c$ is positive,   set $w_c= f(C(u_1), C(o_2), +)$. See Figure~\ref{function}.
\item When $c$ is negative,   set $w_c= f(C(u_2), C(o_1), -)$. See Figure~\ref{function}.
\end{itemize}
\begin{figure}[h]
\begin{center}
\includegraphics[width=75mm]{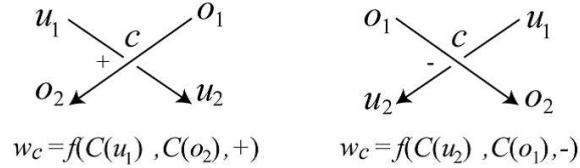}
\caption{The weight of a real-crossing $c$ with respect to an up-down coloring.}
\label{function}
\end{center}
\end{figure}
Let $W_f(D,C)$ be the sum of the weights for all the real-crossings of $D$, that is, 
\[
W_{f}(D,C)= \sum_{c} w_c.
\]
We denote by $\Phi_{f}(D)$ the multi-set 
\[
\big\{ W_{f}(D,C) ~|~ C\in \{ \mbox{$n$-up-down colorings of $D$} \} \big\}.
\]
Theorem~\ref{thm:3} implies that the multi-set $\Phi_{f}(D)$ is useful to detect the necessity of RI\hspace{-1.5pt}I-moves.

\phantom{x}
\noindent {\it Proof of Theorem~\ref{thm:3}.}
Let $D$ and $D'$ be virtual-knot diagrams such that $D'$ is obtained from $D$ by a single RI-move. 
For an $n$-up-down coloring $C$ of $D$,
we set an $n$-up-down coloring $C'$ of $D'$ so that $C$ is corresponding to $C'$ by the bijection defined in the proof of Theorem~\ref{thm:1}.
Then as shown in Figure~\ref{r1r3}, we have 
\[
W_{f}(D,C)-W_{f}(D',C')=\pm f(a,a, \varepsilon)
\]
for some $a \in \mathbb Z_n$ and $\varepsilon \in \{ +, -\}$.
By the condition (0) of an $n$-up-down cocycle, since $f(a,a, \varepsilon)=0$,  $W_{f}(D,C)=W_{f}(D',C')$ holds.

Next let us consider the cases of RI\hspace{-1.5pt}I\hspace{-1.5pt}I-moves. As shown in Figure~\ref{r3}, we have eight RI\hspace{-1.5pt}I\hspace{-1.5pt}I-moves (1)-(8). 
Let $D$ and $D'$ be virtual-knot diagrams such that $D'$ is obtained from $D$ by the  RI\hspace{-1.5pt}I\hspace{-1.5pt}I-move (1) in Figure~\ref{r3}. 
For an $n$-up-down coloring $C$ of $D$,
we set an $n$-up-down coloring $C'$ of $D'$ so that $C$ is corresponding to $C'$ by the bijection defined in the proof of Theorem~\ref{thm:1}.
Then we have 
\[
\begin{array}{l}
W_{f}(D,C)-W_{f}(D',C')\\
\hspace{1cm}=\pm\big\{ (f(a-1, b, -)+f(b+1, c+1, +)+f(a-1, c+2, +))\\
\hspace{2cm}-(f(a-2, b-1, -)+f(b, c+2, +)+f(a, c+1, +))\big\}
\end{array}
\]
for some $a,b,c \in \mathbb Z_n$, see Figure~\ref{r3}.
By the condition (1) of an $n$-up-down cocycle, since the right side of the above equation is equal to $0$, we have $W_{f}(D,C)=W_{f}(D',C')$.
Similarly, in the cases of the other RI\hspace{-1.5pt}I\hspace{-1.5pt}I-moves (2)-(8), by the conditions (2)-(8) of an $n$-up-down cocycle, respectively, we have $W_{f}(D,C)=W_{f}(D',C')$. 

It is obvious that 
the other generalized Reidemeister moves other than the RI\hspace{-1.5pt}I-moves also do not change the weight sum $W_f(D,C)$, see Figure~\ref{vr4} for the VRI\hspace{-1.5pt}V-moves.

As a consequence, we have  a bijection 
$\Phi_f(D) \to \Phi_f(D'), W_f(D,C) \mapsto W_f(D',C')$ such that $W_f(D,C)=W_f(D',C')$
for each of the generalized Reidemeister moves other than the RI\hspace{-1.5pt}I-moves. 
\hfill \qed

\phantom{x}

Now we assume that the $n$-up-down cocycle $f$ is shiftable. 
Then the value $W_{f}(D,C)$ does not depend on the choice of the $n$-up-down coloring $C$, see Figure~\ref{shift} and Corollary~\ref{cor:shift}.
\begin{figure}[h]
\begin{center}
\includegraphics[width=100mm]{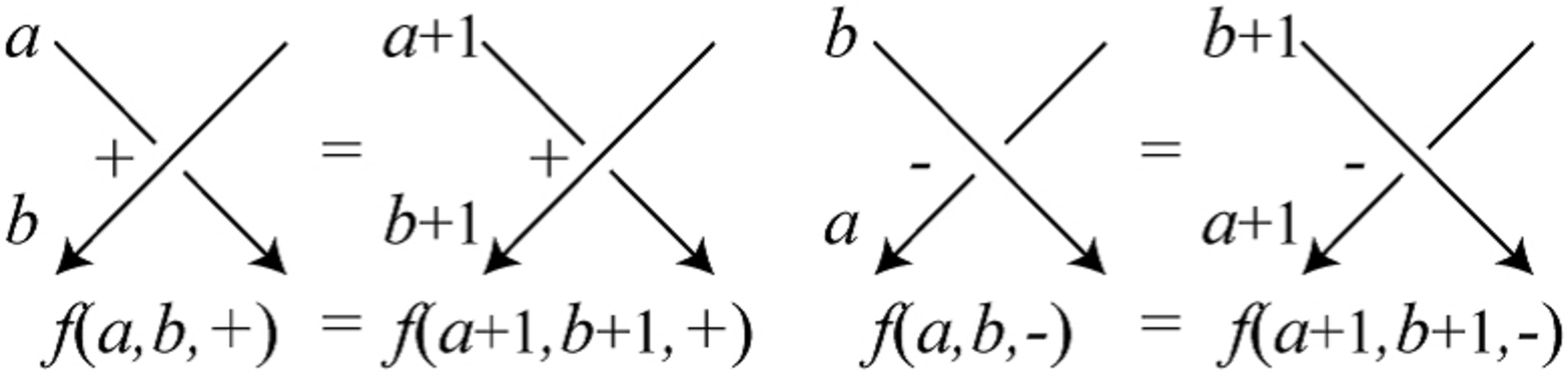}
\caption{The weight of a real-crossing with respect to up-down colorings $C$ and $C+1$}
\label{shift}
\end{center}
\end{figure}
Hence, we denote by $\Phi_{f}^{\rm shift}(D)$ the value $W_{f}(D,C)$ for some $n$-up-down coloring $C$. Note that $\Phi_{f}(D)$ and $\Phi_{f}^{\rm shift}(D)$ are essentially the same.
Next property is obtained as a corollary of Theorem~\ref{thm:3}. 
\begin{corollary}\label{cor:cocycleinv}
Let $D$ and $D'$ be diagrams which represent the same virtual-knot. 
If $\Phi_f^{\rm shift}(D)\not=\Phi_f^{\rm shift}(D')$, then any sequence of generalized Reidemeister moves between $D$ and $D'$ includes at least one RI\hspace{-1pt}I-move.
\end{corollary}

\begin{lemma}\label{lemma:additive}
Let $D_1$ and $D_2$ be virtual-knot diagrams.
For any shiftable up-down cocycle $f$ and for any connected sum $D=D_1 \sharp D_2$ of $D_1$ and $D_2$, the equality 
\[
\Phi_{f}^{\rm shift}(D)=\Phi_{f}^{\rm shift}(D_1)+\Phi_{f}^{\rm shift}(D_2)
\]
holds.
\end{lemma}
\begin{proof}
Let $n$ be an positive integer. 
Let $C$ be an $n$-up-down coloring of $D$.
Then $C$ is  separated into an $n$-up-down coloring $C_1$ of $D_1$ and an $n$-up-down coloring $C_2$ of $D_2$ such that 
$C_i$ $(i=1,2)$ satisfies that $C_i(e)=C(e')$
if $e \cap e' \not = \emptyset$ for $e \in  \mathcal{SA}(D_i)$ and $e' \in  \mathcal{SA}(D)$.
Hence by the definition of the sum of weights, we have  
\[
\Phi_{f}^{\rm shift}(D)= W_{f}(D,C) = W_{f}(D_1,C_1) + W_{f}(D_2,C_2)  =\Phi_{f}^{\rm shift}(D_1) + \Phi_{f}^{\rm shift}(D_2).
\]
\end{proof}

\begin{example}\label{exam:OD}
Let  $\Delta$ and $O$ denote the virtual-knot diagrams shown in Figure~\ref{virtual2}. Note that both of them represent the trivial virtual-knot and $O$ has some orientation.
\begin{figure}[h]
\begin{center}
\includegraphics[width=50mm]{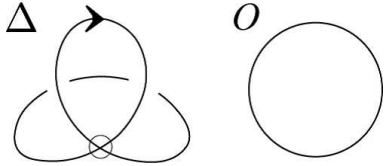}
\caption{Virtual knot diagrams $\Delta$ and $O$.}
\label{virtual2}
\end{center}
\end{figure}
Let $f:\mathbb Z_4 \times \mathbb Z_4 \times \{+,-\} \to \mathbb Z_4$ be the shiftable $4$-up-down cocycle in Example~\ref{exam:cocycle}.
We set the $4$-up-down coloring as shown in Figure~\ref{virtual2}.
Then the sum of the weights of the real-crossings of $\Delta$ is 
\[
f(1, 0 ,-)+f(1, 2,+) = 1+0 =1.
\]
Hence we have $\Phi_{f}^{\rm shift}(\Delta)=1$.
On the other hand, since $O$ has no real-crossing, $\Phi_{f}^{\rm shift}(O)=0$ holds.
Therefore by Corollary~\ref{cor:cocycleinv}, at least one RI\hspace{-1.5pt}I-move is needed  to transform $\Delta$ to $O$.  
\begin{figure}[h]
\begin{center}
\includegraphics[width=50mm]{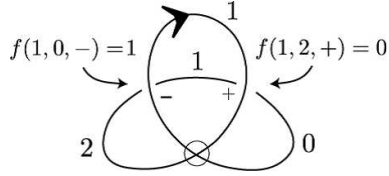}
\caption{$\Phi_{f}^{\rm shift}(\Delta) =1$.}
\label{delta}
\end{center}
\end{figure}
\end{example}

\phantom{x}
\noindent{\it Proof of Theorem~\ref{thm:4}}
Let $D'$ be a virtual-knot diagram which is obtained by taking the connected sum $\Delta \sharp D$ depicted in Figure~\ref{delta-conn}, where we may replace the diagram $\Delta$ with the one obtained after performing   a single RI-move as shown in the lower picture of  Figure~\ref{delta-conn}.
Then $D$ and $D'$ represent the same virtual-knot.
\begin{figure}[h]
\begin{center}
\includegraphics[width=100mm]{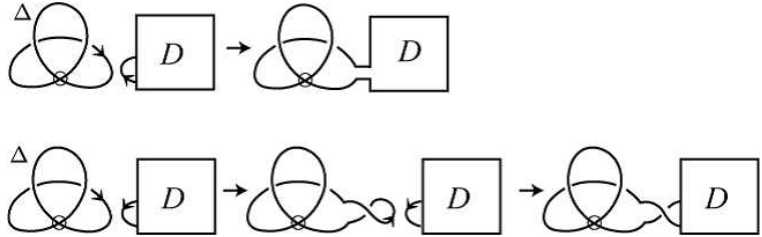}
\caption{A connected sum.}
\label{delta-conn}
\end{center}
\end{figure}

By Lemma~\ref{lemma:additive} and Example~\ref{exam:OD}, 
we have  
\[
\Phi_{f}^{\rm shift}(D')= \Phi_{f}^{\rm shift}(\Delta)+\Phi_{f}^{\rm shift}(D)=1 + \Phi_{f}^{\rm shift}(D) \not = \Phi_{f}^{\rm shift}(D) \mbox{ in $\mathbb Z_4$}.
\]
Therefore by Corollary~\ref{cor:cocycleinv}, at least one RI\hspace{-1.5pt}I-move is needed  to transform $D$ to $D'$.  
\hfill \qed

\phantom{x}

\begin{remark}
The unoriented version of Theorem~\ref{thm:4} also holds. 
To prove this, we choose a shiftable  $n$-up-down cocycle $g$ instead of the $4$-up-down cocycle $f$ in Example~\ref{exam:cocycle} such that for the virtual-knot diagram $\Delta$ in Figure~\ref{virtual2}, $\Phi_{g}^{\rm shift}(\Delta)\not = 0$ and  for any virtual-knot diagram $D$, $\Phi_{g}^{\rm shift}(D)$ does not depend on the orientation of $D$. Then we can see that for the  virtual-knot diagrams $D$ and $D'=\Delta\sharp D$ with any orientations,  $\Phi_{g}^{\rm shift}(D) \not = \Phi_{g}^{\rm shift}(D')$, which implies that as the unoriented virtual-knot diagrams $D$ and $D'$, at least one RI\hspace{-1.5pt}I-move is needed for transformations between them. We leave the detailed proof to the reader. For example, use the $4$-up-down cocycle $g: \mathbb Z_4 \times \mathbb Z_4 \times \{+, -\} \to \mathbb Z_4$ defined by 
\[
g(a,b, \varepsilon) = 
\left\{
\begin{array}{ll}
1 & (a=b\pm 1, \varepsilon =-),\\
0 & (\mbox{otherwise}).\\
\end{array}
\right.
\]
\end{remark}

\section{Questions and future studies}
\begin{itemize}
\item[1.] (To study about a generalization of up-down colorings.)
\end{itemize}
Let $n$ be a nonnegative  integer and  $P$ and $N$ integers.   

An \textit{$(n; P,N)$-up-down coloring} of a virtual-link diagram $D$ is a map $C: \mathcal{SA}(D) \to \mathbb Z_n$ satisfying the following conditions:
For a real-crossing $c$ of $D$, let $u_1$, $u_2$ (resp. $o_1$, $o_2$) be the under-semi-arcs (resp. over-semi-arcs) around $c$  such that the orientation of $D$ points from $u_1$ to $u_2$  (resp. from $o_1$ to $o_2$).  Then 
\begin{itemize}
\item if  $c$ is positive,  
\begin{align*}
C(u_2)=C(u_1)-P \mbox{\  and \ } C(o_2)=C(o_1)+P,
\end{align*}
and 
\item if  $c$ is negative,  
\begin{align*}
C(u_2)=C(u_1)-N \mbox{\  and \ } C(o_2)=C(o_1)+N.
\end{align*} 
\end{itemize}
It is easily seen that $(n;P,N)$-up-down colorings are a generalization of up-down colorings. 
Moreover the numbers of $(n;P,N)$-up-down colorings of virtual-link diagrams can be used to detect the necessity of RI\hspace{-1.5pt}I-moves.
In addition, the argument in Section~\ref{section:cocycleinvariants} is also extended in the cases using $(n; P,N)$-up-down colorings. 
We aim to address these topics including some applications in future work.
(Moreover note that an $(n; P,N)$-up-down coloring is a biquandle coloring if and only if $P=-N$.)

\begin{itemize}
\item[2.] (To find a (co)homology theory related to up-down cocycles.)
\end{itemize}
In \cite{CJKLS}, a quandle cohomology theory was introduced, see also \cite{FRS95, FRS96}.
The $2$- or $3$-cocycles (or $3$- or $4$-cocycles) are used to define a quandle cocycle invariant of an oriented classical-link (or an oriented surface-link).
The cocycle conditions are related to the RI-moves and  an  RI\hspace{-1.5pt}I\hspace{-1.5pt}I-move, see \cite{CJKLS}.

In Section~\ref{section:cocycleinvariants}, we defined an up-down cocycle.
The cocycle conditions are also related to the RI-moves and all the RI\hspace{-1.5pt}I\hspace{-1.5pt}I-moves. 
However we do not know  if there exists a (co)homology theory which are related to up-down cocycles. 
Hence, it might be natural to ask the following question: 
\begin{question}
Is there a cohomology theory whose $2$-cycles are related to the up-down cocycles?  
\end{question}

\begin{itemize}
\item[3.] (To find an application for classical-knots.)
\end{itemize}
In Section~\ref{section:cocycleinvariants}, by our method, we found two trivial virtual-knot diagrams such that at least one RI\hspace{-1.5pt}I-move is needed for transformations between them. For classical-knot diagrams, we give the following question:
\begin{question}
Is there a pair of trivial classical-knot diagrams such that at least one RI\hspace{-1.5pt}I-move are needed for transformations between them and
the necessity of RI\hspace{-1.5pt}I-moves is detected by our method (, but not detected by Hagge's or Manturov's one)?
\end{question}
\noindent If we find such a pair, we might be able to  give a new example satisfying the following theorem:
\begin{theorem}(cf. \cite{hagge})
For any classical-knot diagram $D$, there exists a diagram $D'$ representing the same classical-knot such that any sequence of Reidemeister moves between $D$ and $D'$ includes at least one RI\hspace{-1.5pt}I-move. 
\end{theorem}

\begin{itemize}
\item[4.] (To study about the other generalized Reidemeister moves.)
\end{itemize}
\begin{question}
For each $\mathcal{M}\in \{ RI, RI\hspace{-1.5pt}I, RI\hspace{-1.5pt}I\hspace{-1.5pt}I, VRI, VRI\hspace{-1.5pt}I, VRI\hspace{-1.5pt}I\hspace{-1.5pt}I,  VRI\hspace{-1.5pt}V\}$, is there a pair of trivial virtual-knot diagrams such that at least one $\mathcal{M}$-move is needed for transformations  between them?
\end{question}
When we restrict the moves $\mathcal{M}$ to $\mathcal{M}\in \{ RI, RI\hspace{-1.5pt}I, RI\hspace{-1.5pt}I\hspace{-1.5pt}I, VRI\}$, the statement of the above question is true, see \cite{CG, hagge, manturov, ostlund,  trace}. 
Note that it is easy to construct a pair of trivial virtual-knot diagrams which need at least one VRI-move between them, since the parity of the number of virtual-crossings is unchanged under the generalized Reidemeister moves except for the VRI-move. 
If we find such a pair of trivial virtual-knot diagrams for the other  generalized Reidemeister moves, we might be able to answer the next conjecture:
\begin{conjecture}
For any virtual-knot $K$, there exist two diagrams $D$ and $D'$ of $K$ such that any sequence of generalized Reidemeister moves between $D$ and $D'$ must contain all of the generalized Reidemeister moves.
\end{conjecture}

\begin{figure}[t]
\begin{center}
\includegraphics[width=120mm]{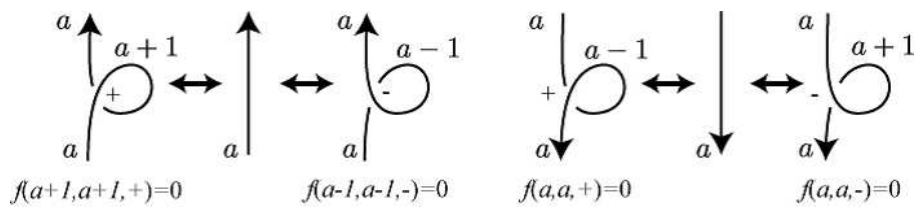}
\caption{The RI-moves.}
\label{r1r3}
\end{center}
\end{figure}
\begin{figure}[t]
\begin{center}
\includegraphics[width=120mm]{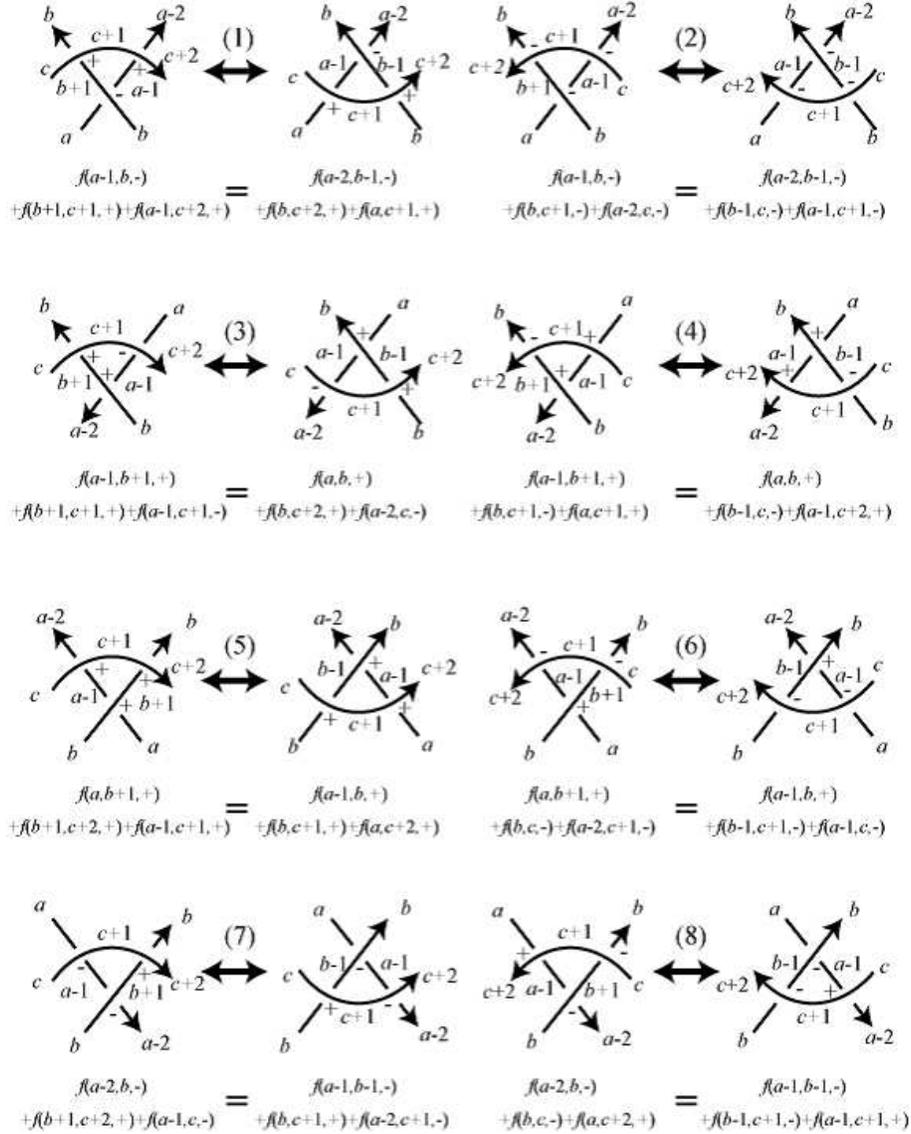}
\caption{The RI\hspace{-1pt}I\hspace{-1pt}I-moves.}
\label{r3}
\end{center}
\end{figure}
\begin{figure}[t]
\begin{center}
\includegraphics[width=50mm]{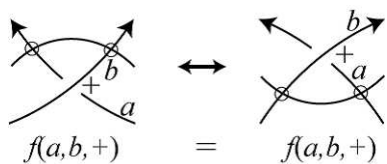}
\caption{VRI\hspace{-1pt}V-moves}
\label{vr4}
\end{center}
\end{figure}



\end{document}